\documentclass[12pt]{article}
\usepackage{graphicx}
\usepackage{amsmath,amsthm,amssymb,enumerate}
\usepackage{euscript,mathrsfs}
\usepackage{color}
\usepackage{dsfont}
\usepackage[left=2cm,right=2cm,top=3.5cm,bottom=3.5cm]{geometry}
\usepackage{color}
\usepackage[framemethod=tikz]{mdframed}
\allowdisplaybreaks

\usepackage{soul}

\catcode`\@=11 \@addtoreset{equation}{section}

\catcode`\@=12

\newtheorem{Theorem}{Theorem}[section]
\newtheorem{Proposition}[Theorem]{Proposition}
\newtheorem{Lemma}[Theorem]{Lemma}
\newtheorem{Corollary}[Theorem]{Corollary}

\theoremstyle{definition}
\newtheorem{Definition}[Theorem]{Definition}

\newtheorem{Remark}[Theorem]{Remark}

\newcommand{\bTheorem}[1]{
	\begin{Theorem} \label{T#1} }
	\newcommand{\eT}{\end{Theorem}}

\newcommand{\bProposition}[1]{
	\begin{Proposition} \label{P#1}}
	\newcommand{\eP}{\end{Proposition}}

\newcommand{\bLemma}[1]{
	\begin{Lemma} \label{L#1} }
	\newcommand{\eL}{\end{Lemma}}

\newcommand{\bCorollary}[1]{
	\begin{Corollary} \label{C#1} }
	\newcommand{\eC}{\end{Corollary}}

\newcommand{\bRemark}[1]{
	\begin{Remark} \label{R#1} }
	\newcommand{\eR}{\end{Remark}}

\newcommand{\bDefinition}[1]{
	\begin{Definition} \label{D#1} }
	\newcommand{\eD}{\end{Definition}}

\newcommand{\Del}{\Delta_x}

\newcommand{\Ds}{\mathbb{D}_x}

\newcommand{\vuB}{\vc{u}_B}

\newcommand{\bFormula}[1]{
	\begin{equation} \label{#1}}
	\newcommand{\eF}{\end{equation}}

\newcommand{\Ov}[1]{\overline{#1}}

\newcommand{\aleq}{\stackrel{<}{\sim}}

\newcommand{\ageq}{\stackrel{>}{\sim}}

\newcommand{\vr}{\varrho}

\newcommand{\vt}{\vartheta}
\newcommand{\vu}{\vc{u}}

\newcommand{\vc}[1]{{\bf #1}}

\newcommand{\Div}{{\rm div}_x}
\newcommand{\Grad}{\nabla_x}

\newcommand{\dx}{\,{\rm d} {x}}

\newcommand{\dt}{\,{\rm d} t }

\newcommand{\intO}[1]{\int_{\Omega} #1 \ \dx}

\newcommand{\D}{{\rm d}}

\newcommand{\ep}{\varepsilon}

\newcommand{\vtB}{\vt_B}

\newcommand{\br}{ \nonumber \\ }

\def\softd{{\leavevmode\setbox1=\hbox{d}%
		\hbox to 1.05\wd1{d\kern-0.4ex{\char039}\hss}}}
\definecolor{Cgrey}{rgb}{0.85,0.85,0.85}
\definecolor{Cblue}{rgb}{0.50,0.85,0.85}
\definecolor{Cred}{rgb}{1,0,0}
\definecolor{fancy}{rgb}{0.10,0.85,0.10}

\newcommand\Cbox[2]{%
	\newbox\contentbox%
	\newbox\bkgdbox%
	\setbox\contentbox\hbox to \hsize{%
		\vtop{
			\kern\columnsep
			\hbox to \hsize{%
				\kern\columnsep%
				\advance\hsize by -2\columnsep%
				\setlength{\textwidth}{\hsize}%
				\vbox{
					\parskip=\baselineskip
					\parindent=0bp
					#2
				}%
				\kern\columnsep%
			}%
			\kern\columnsep%
		}%
	}%
	\setbox\bkgdbox\vbox{
		\color{#1}
		\hrule width  \wd\contentbox %
		height \ht\contentbox %
		depth  \dp\contentbox
		\color{black}
	}%
	\wd\bkgdbox=0bp%
	\vbox{\hbox to \hsize{\box\bkgdbox\box\contentbox}}%
	\vskip\baselineskip%
}

\mdfdefinestyle{MyFrame}{%
	linecolor=black,
	outerlinewidth=1pt,
	roundcorner=5pt,
	innertopmargin=\baselineskip,
	innerbottommargin=\baselineskip,
	innerrightmargin=10pt,
	innerleftmargin=10pt,
	backgroundcolor=white!20!white}

\begin{document}


\title{On strict positivity of the temperature in the Navier--Stokes--Fourier system}

\author{Eduard Feireisl
		\thanks{The work of E.F. was partially supported by the
			Czech Sciences Foundation (GA\v CR), Grant Agreement
			21--02411S. The Institute of Mathematics of the Academy of Sciences of
			the Czech Republic is supported by RVO:67985840. }}

\date{}

\maketitle

\centerline{Institute of Mathematics of the Academy of Sciences of the Czech Republic}

\centerline{\v Zitn\' a 25, CZ-115 67 Praha 1, Czech Republic}

\begin{abstract}
	
We show the Navier--Stokes--Fourier system driven by inhomogeneous 
Dirichlet boundary conditions admits a weak solution with a strictly positive temperature as long as the initial/boundary temperature is bounded below away from zero.

\end{abstract}




\section{Introduction}
\label{i}

In accordance with the basic principles of classical thermodynamics, the \emph{absolute} temperature must remain strictly positive
for any physically admissible process. A rigorous verification of this property for a specific mathematical \emph{model} of a fluid in motion may be, however, rather difficult. One of the frequently used models of viscous and 
heat conducting fluids is the \emph{Navier--Stokes--Fourier system}: 
\begin{align}
\partial_t \vr + \Div (\vr \vu) &= 0, \br 
\partial_t (\vr \vu) + \Div (\vr \vu \otimes \vu) + \Grad p &= 
\Div \mathbb{S} + \vr \vc{g}, \br
\vr c_v \left( \partial_t \vt + \vu \cdot \Grad \vt \right) + \Div \vc{q}
&= \mathbb{S}: \Grad \vu - \vt \frac{\partial p}{\partial \vt} \Div \vu, 
 \label{i1}
 \end{align}
where the viscous stress $\mathbb{S}$ is given by \emph{Newton's rheological law}:
\begin{equation} \label{i2}
	\mathbb{S} = \mu \left(\Grad \vu + \Grad^t \vu - \frac{2}{3} \Div \vu \mathbb{I}   \right) + \eta \Div \vu \mathbb{I},  
\end{equation} 
while the heat flux $\vc{q}$ obeys Fourier's law:
\begin{equation} \label{i3}
	\vc{q} = - \kappa \Grad \vt.
	\end{equation}
	
The state of the fluid at a time $t$ and a spatial position $x$ is characterized by its mass density $\vr = \vr(t,x)$, the (absolute) temperature $\vt = \vt(t,x)$, and the velocity $\vu = \vu(t,x)$.
The system is formally closed by specifying the state equations relating 
the pressure $p = p(\vr,\vt)$ and the specific heat at constant 
volume $c_v = c_v(\vr, \vt)$ to the thermodynamic state variables $(\vr,\vt)$. 

We suppose the fluid is confined to a bounded physical domain $\Omega \subset 
R^3$, and the system of field equations supplemented with the 
boundary conditions 
\begin{equation} \label{i4}
	\vu|_{\partial \Omega} = \vuB,\ 
	\vt|_{\partial \Omega} = \vtB, 
\end{equation}
and the initial conditions 
\begin{equation} \label{i5}
	\vr(0, \cdot) = \vr_0,\ \vt(0, \cdot) = \vt_0,\ \vu(0, \cdot) = \vu_0.
\end{equation}	

Our goal is to establish the minimum principle for the temperature 
\begin{equation} \label{i6}
\inf_{t \in (0,T), x \in \Omega} \vt(t,x) \geq \underline{\vt},\ 
\mbox{where}\ \underline{\vt} 
\ \mbox{depends solely on}\ \min_{\partial \Omega} \vtB, 
\ \min_{\Omega} \vt_0 \ \mbox{and}\ T.
\end{equation}
Relation \eqref{i6} can be seen as an {\it a priori} bound satisfied by any \emph{smooth} solution of the problem. Once the minimum principle is established, we show how to incorporate it in the approximation scheme to obtain 
a \emph{weak solution} to the Navier--Stokes--Fourier system enjoying the same property.

To the best of our knowledge, there are only two results available 
concerning the minimum principle in the class of weak solutions of the Navier--Stokes--Fourier system. Mellet and Vasseur \cite{MeVa3}
(see also Wang and Zuo \cite{WanZuo}) consider the pressure state equation in the form 
\[
p(\vr, \vt) = p_e(\vr) + \vt p_{{\rm th}}(\vr)
\]
adopting the framework of weak solutions introduced in \cite{EF70}. 
Unfortunately, the concept of weak solution developed in \cite{EF70} is  
not strong enough to guarantee the existence of the temperature 
gradient, and, consequently, to establish the weak--strong uniqueness property. A more convenient approach to weak solutions, 
based on introducing the entropy balance as one of the state equations, was introduced in \cite{FeNo6}  and recently developed 
in \cite{ChauFei}, \cite{FeiNovOpen} to accommodate the Dirichlet boundary conditions \eqref{i4}. The pressure state equation 
considered in \cite{FeNo6} takes the form
\begin{equation} \label{i7}
p(\vr, \vt) \approx \vt^{\frac{5}{2}} P \left( \frac{\vr}{ \vt^{\frac{3}{2}}}\right) + a \vt^4,\ a > 0,
\end{equation}
where $a \vt^4$ is the so--called radiation pressure that plays a crucial role in the existence theory. Motivated by \cite{FeNo6}, 
Baer a Vasseur \cite{BaeVas} established the temperature minimum principle for the pressure state equation in the form 
\[
p(\vr, \vt) \approx \vt^{\frac{5}{2}} P \left( \frac{\vr}{ \vt^{\frac{3}{2}}}\right).
\]
Unfortunately, as pointed out in \cite{BaeVas}, this approach does not allow to incorporate the radiation pressure $a \vt^4$ so far 
indispensable to the existence theory.

We show validity of the {\it a priori} bound \eqref{i6} for a large class of state equations including 
\eqref{i7} subject to suitable restrictions compatible with the Third law of thermodynamics. First, 
in Section \ref{T}, we derive a general temperature minimum principle as {\it a priori} bound for a vast class of state equations 
including those relevant to the theory developed in \cite{FeiNovOpen}. In Section \ref{p}, we show how to incorporate the minimum principle in the approximate scheme used in \cite{FeiNovOpen}.
Accordingly, 
we show the existence of a weak solution with the temperature bounded below away from zero in terms of the initial/boundary data. The paper is concluded by 
final remarks in Section \ref{C}. 

\section{Temperature minimum principle as {\it a priori} bound}
\label{T} 

Our goal is to identify the class of state equations, for which the Navier--Stokes--Fourier system admits 
the temperature minimum principle \eqref{i6}. The key quantities are 
\begin{align}
\mbox{the bulk viscosity coefficient}\ \eta &= \eta(\vt),\br 
\mbox{the heat conductivity coefficient}\ \kappa &= \kappa(\vt),\br 
\mbox{the state equations}\ p = \vt^{\frac{5}{2}} P \left( \frac{\vr}{ \vt^{\frac{3}{2}}}\right) + \frac{a}{3} \vt^4,\  
c_v &= \frac{9}{4}\frac{ \frac{5}{3} P(Z) - P'(Z) Z}{Z} + \frac{4a}{\vr} \vt^3,\ \mbox{where}\ Z= \frac{\vr}{ \vt^{\frac{3}{2}}},
\label{T1}
\end{align}
cf. \cite[Chapter 4]{FeiNovOpen}. Note that \eqref{T1} includes the standard Boyle--Mariotte law corresponding to $P(Z) = Z$.

\subsection{Desired properties of the state equation}

We suppose $\kappa = \kappa(\vt)$ is a continuously differentiable function of $\vt$, 
\begin{equation} \label{T2}
0 < \underline{\kappa} \ \vt \kappa(\vt) \leq \mathcal{K}(\vt) \equiv \int_0^\vt \kappa(s) \ \D s 
\leq \Ov{\kappa}\ \vt \kappa(\vt) \ \mbox{for any}\ \vt \geq 0
\end{equation}
for some positive constants $\underline{\kappa}$, $\Ov{\kappa}$.

As for the state equation, first observe 
\begin{equation} \label{T3}
0 < \underline{e} \ \vr c_v(\vr, \vt) \leq 
\frac{\partial p(\vr, \vt)}{\partial \vt} \leq 	\underline{e} \ \vr c_v(\vr, \vt)
\ \mbox{for all}\ \vr, \vt > 0
	\end{equation}
whenever $a \geq 0$. Here, specifically, $\underline{e} = \frac{1}{3}$, $\Ov{e} = \frac{2}{3}$.

The last property we desire imposes some restrictions on the function $P$ in \eqref{T1}:
\begin{equation} \label{T4}
\vr c_v (\vr, \vt) \leq C_v (\vt) \ \mbox{for all}\ \vr, \vt > 0, 
\end{equation}
where $C_v$ depends only on the temperature.
Indeed, to make \eqref{T4} compatible with \eqref{T1}, we assume 
\begin{equation} \label{T5}
0 < \frac{5}{3}P(Z) - P'(Z)Z \leq \Ov{P} \ \mbox{for all}\ Z > 0, 
	\end{equation}
yielding 
\[
\vr c_v (\vr, \vt) \leq \frac{9}{4} \Ov{P} \vt^{\frac{3}{2}} + 4 a \vt^3.
\]

Property \eqref{T5} is intimately related to the Third law of thermodynamics, 
namely,
\[
s(\vr, \vt) \to 0 \ \mbox{as}\ \vt \to 0,
\]
where $s$ denotes the entropy, cf. Belgiorno 
\cite{BEl1}, \cite{BEL2}. Indeed the entropy associated to \eqref{T1} by Gibbs' relation
\[
\vt Ds = D e + p D\left(\frac{1}{\vr}\right)
\]
reads 
\[
s(\vr, \vt) = S(Z) + \frac{4a}{3} \frac{\vt^3}{\vr},\ 
S'(Z) = - \frac{3}{2} \frac{ \frac{5}{3}P(Z) - P'(Z)Z  }{Z^2},\ Z = \frac{\vr}{\vt^{\frac{3}{2}}};
\]
whence $S(Z) \to 0$ as $Z \to \infty$ whenever \eqref{T5} holds.

\subsection{Minimum principle}

Our goal is to apply the standard minimum principle for parabolic equations to the internal energy 
balance
\[
\vr c_v \left( \partial_t \vt + \vu \cdot \Grad \vt \right) - \Div (\kappa(\vt) \Grad \vt)
= \mathbb{S}: \Grad \vu - \vt \frac{\partial p}{\partial \vt} \Div \vu,\ 
\vt(0, \cdot) = \vt_0,\ \vt|_{\partial \Omega} = \vtB.
\]

First, introducing $\Theta = \int_0^\vt \kappa(s) \ \D s$ we obtain
\begin{equation} \label{T6}
\partial_t \Theta + \vu \cdot \Grad \Theta  - \frac{\kappa(\vt)}{\vr c_v(\vr, \vt)} \Del \Theta 
\geq \frac{ \kappa (\vt) \eta(\vt) }{\vr c_v} |\Div \vu|^2 - \Theta \frac{\kappa(\vt)\vt}{\Theta} (\vr c_v(\vr, \vt))^{-1}\frac{\partial p}{\partial \vt} \Div \vu
	\end{equation}

Next, using the hypotheses \eqref{T2}--\eqref{T4} we deduce 
\begin{align}
\frac{ \kappa (\vt) \eta(\vt) }{\vr c_v(\vr, \vt)} |\Div \vu|^2 - &\Theta \frac{\kappa(\vt)\vt}{\Theta} (\vr c_v(\vr, \vt))^{-1}\frac{\partial p}{\partial \vt} \Div \vu
\br &\geq \frac{ \kappa (\vt) \eta(\vt) }{C_v(\vt)} |\Div \vu|^2 - \Theta \frac{\Ov{\kappa}}{\underline{\kappa}} \underline{e} |\Div \vu|
\nonumber
\end{align}
Thus, imposing the final stipulation 
\begin{equation} \label{T7}
\frac{ \kappa (\vt) \eta(\vt) }{C_v(\vt)} \geq \Lambda > 0 \ \mbox{for all}\ \vt > 0,
\end{equation}
we may infer infer that 
\begin{equation} \label{T8}
	\partial_t \Theta + \vu \cdot \Grad \Theta  - \frac{\kappa(\vt)}{\vr c_v(\vr, \vt)} \Del \Theta 
	\geq  - M \Theta^2,\ \Theta(0, \cdot) = \mathcal{K}(\vt_0),\ \Theta|_{\partial \Omega} = 
	\mathcal{K}(\vtB), 
\end{equation}
where $M$ depends only on the structural properties of the state equations. 

To conclude, we introduce a suitable subsolution to \eqref{T8}, specifically $Y = Y(t)$,
\[
\frac{\D}{\dt} Y(t) = - M Y^2,\ Y(0) = Y_0 = \frac{1}{2}\min\{ \min_\Omega \mathcal{K}(\vt_0),\ 
\min_\Omega	\mathcal{K}(\vtB) \},
\]
meaning
\begin{equation} \label{T9}
Y(t) = \left( \frac{1}{Y_0} + M t\right)^{-1}.
\end{equation}
Using the standard parabolic comparison theorem we may infer that 
\[
\Theta (t,x) > Y(t) \ \mbox{for all}\ t \geq 0,\ x \in \Omega.
\]

We have shown the following result. 

\begin{Proposition} [{\bf Temperature minimum principle}] \label{TT1}

Let the functions $\kappa, \eta \in C^1( 0, \infty)$, 
$p \in C^1(0, \infty)^2$, $c_v \in C(0, \infty)^2$ satisfy the structural restrictions \eqref{T2}--\eqref{T4}. In addition, 
suppose there is $\Lambda$ such that 
\begin{equation} \label{cond}
\frac{\kappa(\vt) \eta(\vt)}{\vr c_v(\vr, \vt)} \geq \Lambda > 0 
\ \mbox{for all}\ \vr > 0, \vt > 0.
\end{equation}

Let $\Omega \subset R^3$ be a bounded domain of class
$C^{2 + \nu}$. 
Let 
\[
\vt \in C([0, T] \times \Ov{\Omega}),\ 
\partial_t \vt \in L^2((0,T) \times \Omega),\ 
\vt \in L^2(0,T; W^{2,2}(\Omega)), \vt > 0,
\]
satisfy
\begin{align}
\vr c_v (\vr, \vt) \left( \partial_t \vt + \vu \cdot \Grad \vt \right) - \Div (\kappa(\vt) \Grad \vt)
&= \mathbb{S}: \Grad \vu - \vt \frac{\partial p (\vr, \vt)}{\partial \vt} \Div \vu \ \mbox{a.a. in}\ (0,T) \times \Omega, \br 
\vt(0, \cdot) &= \vt_0,\ \vt|_{\partial \Omega} = \vtB,
\nonumber
\end{align}
where $\vr \in C([0,T] \times \Ov{\Omega})$, $\vr > 0$, 
$\vu \in C([0,T] \times \Ov{\Omega}; R^d)$, $\Grad \vu \in C([0,T] \times \Ov{\Omega}; R^{d \times d})$.

Then there is $M$ depending only on the structural constants 
$\underline{\kappa}$, $\Ov{\kappa}$, $\underline{e}$, $\Ov{e}$, $\Lambda$ such that 
\begin{align}
\vt(t, x) &> \mathcal{K}^{-1} \left( \left( \frac{1}{Y_0} + M t \right)^{-1} \right),\ 
Y_0 = \frac{1}{2}\min\{ \min_\Omega \mathcal{K}(\vt_0),\ 
\min_{\partial \Omega \times (0, \tau)} 	\mathcal{K}(\vtB) \}, \br 
&\mbox{for a.a.}\ (t,x) \in (0,T) \times \Omega.
\nonumber
\end{align}

	\end{Proposition}
	
\begin{proof}
	
	Repeating the arguments presented at the beginning of this section
we arrive at the inequality 	
\begin{equation} \label{T10}
\partial_t V(t) + \vu \cdot \Grad V - \frac{\kappa(\vt)}{\vr c_v(\vr, \vt)} \Del V 
\geq - M (\Theta + Y) V \ \mbox{a.a. in}\ (0,T) \times \Omega, 
\end{equation}
for $V = \mathcal{K} (\vt ) - Y$, $Y$ given by \eqref{T9}, satisfying 
\[
V(0, \cdot) > 0 ,\ V|_{\partial \Omega} > 0 .
\]
Consequently, the proof of Proposition \ref{TT1} reduces to showing $V \geq 0$.

Since $Y$ and $\Theta = \mathcal{K}(\vt)$ are bounded, there is a constant $K$ such that 
\[
	\partial_t V(t) + \vu \cdot \Grad V - \frac{\kappa(\vt)}{\vr c_v(\vr, \vt)} \Del V 
	\geq - K V \ \mbox{a.a. in}\ (0,T) \times \Omega, 
\]
as long as $V \geq 0$. Thus introducing $\chi = \exp(Kt) V$ we get 
\[
\partial_t \chi(t) + \vu \cdot \Grad \chi - \frac{\kappa(\vt)}{\vr c_v(\vr, \vt)} \Del \chi 
\geq 0 \ \mbox{a.a. in}\ (0,T) \times \Omega, \ \chi(0, \cdot) > 0, 
\ \chi|_{(0,T) \times \partial \Omega } > 0 
\]
as long as $\chi \geq 0$. If $\chi$ was regular, the result would follow from the Strong minimum principle stating $\chi > 0$ as long as $\chi \geq 0$.

If $\chi$ is not regular but still a strong solution, we get 
\begin{equation} \label{T11}
\partial_t \chi(t) + \vu \cdot \Grad \chi - \frac{\kappa(\vt)}{\vr c_v(\vr, \vt)} \Del \chi  = g \in \ L^2((0,T) \times \Omega),\ 
g \geq 0, 
\end{equation}
as long as $\chi \geq 0$, 
\begin{equation} \label{T12}
	\chi(0, \cdot) > 0, 
\ \chi|_{(0,T) \times \partial \Omega } > 0. 
\end{equation}
Thus desired conclusion follows by approximating $\vu$, the diffusion 
coefficient $\frac{\kappa(\vt)}{\vr c_v(\vr, \vt)}$, and $g$ by smooth functions. Indeed the problem \eqref{T11}, \eqref{T12} admits a unique solution in the class  
\[
\partial_t \chi \in L^2((0,T) \times \Omega),\ 
\chi \in L^2(0,T; W^{2,2}(\Omega));
\] 
whence the result can be extended to strong solutions.

	\end{proof}	

\section{Positivity of the temperature in the weak solutions to the Navier--Stokes--Fourier system}
\label{p}

Our goal is to show how the temperature minimum principle established in the preceding section can be incorporated in the construction 
of weak solutions to the Navier--Stokes--Fourier system. 

\subsection{Weak solutions}
\label{WS} 

A suitable concept of \emph{weak} solution to the Navier--Stokes--Fourier system \eqref{i1}--\eqref{i3}, with the Dirichlet boundary data \eqref{i4} was introduced in 
\cite{ChauFei}, \cite[Chapter 12]{FeiNovOpen}. For the sake of simplicity, we 
suppose $\vuB = 0$. Note, however, that the role of the velocity in the subsequent analysis is marginal and more complicated boundary behaviour of $\vu$ can be handled.

First, extend the boundary data $\vtB$ inside $\Omega$. We suppose the extension is continuously differentiable and satisfies
\begin{equation} \label{p1}
\inf_{(0,T) \times \Omega} \vtB \geq 
\underline{\vt} > 0.
 \end{equation}
The temperature and velocity components of a weak solution belong to the class 
\begin{equation} \label{p2} 
 \vu \in L^2((0,T); W^{1,2}_0(\Omega; R^3)),\ 
 (\vt - \vtB) \in L^2((0,T); W^{1,2}_0(\Omega)),\ 
 \vt > 0 \ \mbox{a.a. in}\ (0,T) \times \Omega.
\end{equation}
Here and hereafter, we always tacitly assume $\vu$ has be extended to be zero outside $\Omega$.

The weak formulation is based on the entropy rather than internal energy balance. Specifically, 
\begin{align} 
\partial_t b(\vr) + \Div ((b(\vr) \vu ) + \Big(b'(\vr) \vr - b(\vr) \Big) \Div \vu &= 0\ \mbox{in}\ \mathcal{D}'((0,T) \times R^3))
\label{p3}\\
&\mbox{for any}\  b \in C^2[0,\infty),\ b'' \in C_c[0, \infty); \br 
\label{p4}
\partial_t (\vr \vu) + \Div (\vr \vu \otimes \vu) + \Grad p &= 
\Div \mathbb{S} + \vr \vc{g} \ \mbox{in}\ \mathcal{D}'((0,T) \times \Omega; R^3); \\
\partial_t (\vr s) + \Div (\vr s \vu) - \Div \left(\frac{\kappa}{\vt} \Grad \vt \right) &= \sigma \ \mbox{in}\ \mathcal{D}'((0,T) \times \Omega), \label{p5} \\
&\mbox{where} \ \sigma \ \mbox{is a measure on}\ [0,T] \times \Ov{\Omega}, \br
\sigma &\geq \frac{1}{\vt} 
\left( \mathbb{S}(\Grad \vu) : \Grad \vu - \frac{\kappa |\Grad \vt|^2 }{\vt} \right). \label{p6}
\end{align}	
Possible loss of information due to the inequality in \eqref{p6} 
is compensated by imposing the \emph{ballistic energy inequality}:
\begin{align}
\frac{\D }{\dt} &\intO{ \left( \frac{1}{2} \vr |\vu|^2 + \vr e(\vr, \vt)  - \vtB \vr s(\vr, \vt) \right) }	+ \int_{\Ov{\Omega}} \vtB 
\ \D \sigma \br
&\leq - \intO{ \left( \vr s (\vr, \vt) \partial_t \vtB + \vr s (\vr, \vt) \vu \cdot \Grad \vtB - \frac{{\kappa}}{ \vt} \Grad \vt \cdot \Grad \vtB  - \vr \vc{g} \cdot \vu       \right)      }
\ \mbox{in}\ \mathcal{D}'(0,T). 
	\label{p7a}
	\end{align}
The reader may consult \cite[Chapter 12]{FeiNovOpen} for the physical background of the above formulation as well as the properties of the weak solutions.  

\subsection{Constitutive relations}

The mathematical theory developed in \cite{FeiNovOpen} is based on certain physically grounded restrictions imposed on the state equations as well as the transport coefficients. Specifically, we suppose (cf. \eqref{T1}):
\begin{align}
p(\vr, \vt) &= 	\vr \vt \frac{ P \left( Z \right) }{Z}   + \frac{a}{3} \vt^4, \  
e(\vr, \vt) = \frac{3}{2} \vt \frac{ P \left( Z \right) }{Z} + \frac{a}{\vr} \vt^4     , \ a > 0,\ Z = \frac{\vr}{\vt^{\frac{3}{2}}  }, \label{p7} \\ 
P &\in C^1[0, \infty) ,\ P'(Z) > 0 \ \mbox{for all}\ Z \geq 0,\ 
{ \frac{5}{3} P(Z) - P'(Z) Z} > 0 \ \mbox{for all}\ Z > 0,
\label{p8} \\ 
s(\vr, \vt) &= S(Z) + \frac{4a}{3} \frac{\vt^3}{\vr},\ 
S'(Z) = - \frac{3}{2} \frac{ \frac{5}{3} P(Z) - P'(Z) Z}{Z^2},\ 
\label{p9}\\ 
\ \frac{P(Z)}{Z^{\frac{5}{3}}} &\to p_\infty > 0,  
\ 
S(Z) \to 0 \ \mbox{as}\ Z \to \infty, \label{p10}
	\end{align}  
together with 
\begin{align} 
\mu,\ \eta \in C^1[0, \infty),\ 
0 < \underline{\mu}(1 + \vt) &\leq \mu(\vt) \leq 
\Ov{\mu}(1 + \vt),\ |\mu' (\vt)| \aleq 1,
	\label{p10a}\\ 
	0  &\leq \eta(\vt) \leq 
	\Ov{\eta}(1 + \vt), \label{p11}\\
0 < \underline{\kappa} (1 + \vt^\beta) &\leq 
\kappa (\vt) \leq \Ov{\kappa} (1 + \vt^\beta),\ \beta > 6.
\label{p12}	
	\end{align}
	
In accordance with \eqref{p7}, we have 
\begin{equation} \label{p13}	
c_v (\vr, \vt) = \frac{\partial e(\vr, \vt)}{\partial \vt} = 
\frac{9}{4}	\frac{ \frac{5}{3} P(Z) - P'(Z) Z}{Z} + \frac{4a}{\vr} \vt^3, \ Z = \frac{\vr}{\vt^{\frac{3}{2}}  }, 
\end{equation}
cf. \eqref{T1}.

Going back to Proposition \ref{TT1}, in particular the hypotheses
\eqref{T2}--\eqref{T4}, 
we observe that \eqref{T2} holds whenever $\vt > 1$ and $\kappa$ satisfies
\eqref{p12}. As for the limit regime $\vt \to 0$, we replace the lower bound in \eqref{p12} by an assumption relevant to gases, namely $\kappa (\vt) \approx \sqrt{\vt}$. Accordingly, we modify \eqref{p12} 
postulating
\begin{equation} \label{p14}
0 < \underline{\kappa}\left( \vt^{\frac{1}{2}} + \vt^\beta\right)	
\leq \kappa (\vt) \leq \Ov{\kappa} \left( \vt^{\frac{1}{2}} + \vt^\beta\right),\ \beta > 6.
\end{equation}

As stated in \eqref{T5}, the properties \eqref{T3}, \eqref{T4} 
are guaranteed by imposing an extra hypothesis on the structural function $P$, namely 
\begin{equation} \label{p15}
0 < \frac{5}{3} P(Z) - P'(Z) Z < \Ov{P} \ \mbox{for all}\ Z > 0,
\end{equation}
yielding 
\begin{equation} \label{p16}
	\vr c_v (\vr, \vt) \leq \frac{9}{4} \Ov{P} \vt^{\frac{3}{2}} + 
	4a \vt^3.
\end{equation}

Finally, condition \eqref{cond} reads 
\[
(\vt^{\frac{1}{2}}  + \vt^\beta ) \eta (\vt) \ageq \vt^{\frac{3}{2}} 
+ \vt^3. 
\]
Consequently, we replace \eqref{p11} by a more restrictive stipulation 
\begin{equation} \label{p19}
0 < \underline{\eta} \min\{ 1;\vt \} \leq \eta (\vt) \leq \Ov{\eta} (1 + \vt).
\end{equation}

\subsection{Approximation scheme}

The existence of global--in--time weak solutions to the Navier--Stokes--Fourier system was proved in \cite[Chapter 12]{FeiNovOpen} by means of a multi--level approximation scheme: 
\begin{align} 
\partial_t \vr + \Div (\vr \vu) &= \ep \Del \vr,\ \Grad \vr \cdot \vc{n}|_{\partial \Omega} = 0; 	
	\label{p20} \\ 
\partial_t (\vr \vu) + \Div (\vr \vu \otimes \vu) + \Grad p_\delta (\vr) + \Grad p(\vr, \vt) &= \Div \mathbb{S}(\vt, \Grad \vu) + \vr \vc{g} + \ep \Del (\vr \vu),\ \vu|_{\partial \Omega} = 0;	
	\label{p21}\\ 
\vr c_v(\vr, \vt) \Big(\partial_t \vt + \vu \cdot \Grad \vt \Big) - 
\Div(\kappa (\vt) \Grad \vt) &= \mathbb{S}(\vt, \Grad \vu) : \Ds \vu 
- \vt \frac{\partial p(\vr, \vt)}{\partial \vt} \Div \vu \br 
&+ \ep |\Grad \vr|^2 \left( 2 \frac{\partial e(\vr, \vt)}{\partial \vr} + \vr \frac{\partial^2 e(\vr, \vt)}{\partial \vr^2} \right),\ 
\vt|_{\partial \Omega} = \vtB,
\label{p22}	
	\end{align} 
where 
\begin{equation} \label{p23}
	p_\delta (\vr) = \delta \vr^\Gamma,\ \delta > 0, \ \Gamma >> 1.
	\end{equation}
In \cite[Chapter 12]{FeiNovOpen}, the momentum equation \eqref{p21} is solved by means of a Galerkin type approximation but this is not essential here. 
The system \eqref{p20}--\eqref{p22} is nothing other than the artificial viscosity approximation of the equation of continnuity and momentum equation, 
supplemented with a slightly modified internal energy balance.

We show that \eqref{p20}--\eqref{p22} is a consistent approximation of the system \eqref{p3}--\eqref{p6}. To this end, we derive the approximate entropy and ballistic energy balance and show they are satisfied with a consistency error vanishing for $\ep \to 0$, $\delta \to 0$. 

Let us start with the approximate entropy balance. After a bit lengthy 
but straightforward manipulation, we deduce from \eqref{p20}, \eqref{p22}: 
\begin{align}
	\partial_t (\vr s(\vr, \vt)) &+ \Div (\vr s(\vr, \vt) \vu) - 
	\Div \left(\frac{\kappa}{\vt} \Grad \vt \right) = 
	\frac{1}{\vt} \left( \mathbb{S}(\vt, \Grad \vu): \Ds \vu + 
	\frac{\kappa |\Grad \vt|^2}{\vt}\right) \br
	&+ \frac{\ep}{\vt} |\Grad \vr|^2 \left( 2 \frac{\partial e(\vr, \vt)}{\partial \vr} + \vr \frac{\partial^2 e(\vr, \vt)}{\partial \vr^2} \right) + \ep \Del \vr \left(\vr \frac{\partial s(\vr, \vt)}{\partial \vr} + s(\vr, \vt)\right).
	\label{p24}
	\end{align}
	
Similarly, we compute 
\begin{align}
	\partial_t (\vr e(\vr, \vt)) &+ \Div (\vr e(\vr, \vt) \vu) - 
	\Div \left(\kappa \Grad \vt \right) = \mathbb{S}(\vt, \Grad \vu):\Ds \vu - p(\vr, \vt) \Div \vu \br 
&+ \ep	|\Grad \vr|^2 \left( 2 \frac{\partial e(\vr, \vt)}{\partial \vr} + \vr \frac{\partial^2 e(\vr, \vt)}{\partial \vr^2} \right) 
+\ep \Del \vr \left( e(\vr, \vt ) + \vr \frac{\partial e(\vr, \vt)}{\partial \vr }\right).
\label{p25}
\end{align}

Finally, multiplying \eqref{p21} by $\vu$ and integrating by parts, we get 
\begin{align} 
\frac{\D}{\dt} \intO{\left( \frac{1}{2} \vr |\vu|^2 + \frac{\delta}{\Gamma - 1} \vr^\Gamma \right) } &+ \ep 
\intO{ \Big( \vr |\Grad \vu|^2 + \delta \Gamma \vr^{\Gamma - 2}|\Grad \vr|^2 \Big) } = \intO{ \vr \vc{g} \cdot \vu } \br
&+ \intO{ \Big(p(\vr, \vt) \Div \vu - \mathbb{S}(\vt, \Grad \vu):\Ds \vu  \Big) }. \label{p26}
\end{align}	

Integrating \eqref{p25} and adding the resulting expression to 
\eqref{p26} yields
\begin{align}
	\frac{\D}{\dt} &\intO{\left( \frac{1}{2} \vr |\vu|^2 + \vr e(\vr, \vt) +  \frac{\delta}{\Gamma - 1} \vr^\Gamma \right) } + 
	\ep \intO{ \Big( \vr |\Grad \vu|^2 + \delta \Gamma \vr^{\Gamma - 2}|\Grad \vr|^2 \Big) } \br
	&=  \intO{ \vr \vc{g} \cdot \vu } + \int_{\partial \Omega} 
	\kappa (\vtB) \Grad \vt \cdot \vc{n} \ \D S_x \br&+ \ep \intO{ \left[ 	|\Grad \vr|^2 \left( 2 \frac{\partial e(\vr, \vt)}{\partial \vr} + \vr \frac{\partial^2 e(\vr, \vt)}{\partial \vr^2} \right) 
		+ \Del \vr \left( e(\vr, \vt ) + \vr \frac{\partial e(\vr, \vt)}{\partial \vr }\right) \right]
	}.
	\label{p27}
	\end{align}
Integrating by parts the last integral, we obtain 
\begin{align}
&\intO{ \left[ 	|\Grad \vr|^2 \left( 2 \frac{\partial e(\vr, \vt)}{\partial \vr} + \vr \frac{\partial^2 e(\vr, \vt)}{\partial \vr^2} \right) 
	+ \Del \vr \left( e(\vr, \vt ) + \vr \frac{\partial e(\vr, \vt)}{\partial \vr }\right) \right]
} \br 
&= - \intO{ \Grad \vr \cdot \Grad \vt \left( \frac{\partial e(\vr, \vt )}{\partial \vt} + \vr \frac{\partial^2 e(\vr, \vt)}{\partial \vr
	\partial \vt }\right)} 
\nonumber
\end{align}
Thus we conclude
\begin{align}
		\frac{\D}{\dt} &\intO{\left( \frac{1}{2} \vr |\vu|^2 + \vr e(\vr, \vt) +  \frac{\delta}{\Gamma - 1} \vr^\Gamma \right) } + 
	\ep \intO{ \Big( \vr |\Grad \vu|^2 + \delta \Gamma \vr^{\Gamma - 2}|\Grad \vr|^2 \Big) } \br
	&=  \intO{ \vr \vc{g} \cdot \vu } + \int_{\partial \Omega} 
	\kappa (\vtB) \Grad \vt \cdot \vc{n} \ \D S_x
	- \ep  \intO{ \Grad \vr \cdot \Grad \vt \left( \frac{\partial e(\vr, \vt )}{\partial \vt} + \vr \frac{\partial^2 e(\vr, \vt)}{\partial \vr
			\partial \vt }\right)}.
	\label{p28}
	\end{align}
	
Now, using \eqref{p24}, \eqref{p28}, we obtain the approximate ballistic energy inequality 
\begin{align}
	\frac{\D}{\dt} &\intO{\left( \frac{1}{2} \vr |\vu|^2 + \vr e(\vr, \vt) - \vtB \vr s(\vr, \vt) +  \frac{\delta}{\Gamma - 1} \vr^\Gamma \right) } + 
	\ep \intO{ \Big( \vr |\Grad \vu|^2 + \delta \Gamma \vr^{\Gamma - 2}|\Grad \vr|^2 \Big) } \br
	+&\intO{\frac{\vtB}{\vt}\left( \mathbb{S}(\vt, \Grad \vu): \Ds \vu + 
		\frac{\kappa |\Grad \vt|^2}{\vt}\right)  } +\ep \intO{ \frac{\vtB}{\vt} |\Grad \vr|^2 \left( 2 \frac{\partial e(\vr, \vt)}{\partial \vr} + \vr \frac{\partial^2 e(\vr, \vt)}{\partial \vr^2} \right)} \br	
	= &-   \intO{ \left( \vr s (\vr, \vt) \partial_t \vtB + \vr s (\vr, \vt) \vu \cdot \Grad \vtB - \frac{{\kappa}}{ \vt} \Grad \vt \cdot \Grad \vtB  - \vr \vc{g} \cdot \vu       \right)      }
	\br
	&
	- \ep  \intO{ \Grad \vr \cdot \Grad \vt \left( \frac{\partial e(\vr, \vt )}{\partial \vt} + \vr \frac{\partial^2 e(\vr, \vt)}{\partial \vr
			\partial \vt }\right)} 
	+ \ep \intO{\Grad \vtB \cdot \Grad \vr \left(s(\vr, \vt) + 
		\vr \frac{\partial s(\vr, \vt)}{\partial \vr}  \right) }	
		\br& + \ep \intO{\vtB \Grad \vr \cdot 
			\Grad \vt \left(\vr \frac{\partial^2 s(\vr, \vt)}{\partial \vr \partial \vt} + \frac{\partial s(\vr, \vt)}{\partial \vt}\right)}\br  &+ \ep 
			\intO{\vtB |\Grad \vr |^2 \left( 2 \frac{\partial s(\vr, \vt)}{\partial \vr} + \vr \frac{\partial^2 s(\vr, \vt)}{\partial \vr^2} \right) } .
	\label{p29}
\end{align}

Finally, by virtue of Gibbs' relation, 
\[
\frac{\partial s(\vr, \vt)}{\partial \vr } - 
\frac{1}{\vt} \frac{\partial e(\vr, \vt)}{\partial \vr } = - \frac{1}{\vt} \frac{p(\vr, \vt)}{\vr^2},
\]
and, similarly, 
\[
\vr \frac{\partial^2 s(\vr, \vt)}{\partial \vr^2} - 
\frac{\vr}{\vt} \frac{\partial^2 e(\vr, \vt)}{\partial \vr^2} = 
- \frac{1}{\vr \vt} \frac{\partial p(\vr, \vt)}{\partial \vr} + 
2 \frac{1}{\vt} \frac{p(\vr, \vt)}{\vr^2}. 
\]
Consequently, the inequality \eqref{p29} can be rewritten in the form	 
\begin{align}
	\frac{\D}{\dt} &\intO{\left( \frac{1}{2} \vr |\vu|^2 + \vr e(\vr, \vt) - \vtB \vr s(\vr, \vt) +  \frac{\delta}{\Gamma - 1} \vr^\Gamma \right) } + 
	\ep \intO{ \Big( \vr |\Grad \vu|^2 + \delta \Gamma \vr^{\Gamma - 2}|\Grad \vr|^2 \Big) } \br
	+&\intO{\frac{\vtB}{\vt}\left( \mathbb{S}(\vt, \Grad \vu): \Ds \vu + 
		\frac{\kappa |\Grad \vt|^2}{\vt}\right)  } +\ep \intO{ \frac{\vtB}{\vr \vt} |\Grad \vr|^2 \frac{\partial p(\vr, \vt)}{\partial \vr}} \br	
	= &-   \intO{ \left( \vr s (\vr, \vt) \partial_t \vtB + \vr s (\vr, \vt) \vu \cdot \Grad \vtB - \frac{{\kappa}}{ \vt} \Grad \vt \cdot \Grad \vtB  - \vr \vc{g} \cdot \vu       \right)      }
	\br
	&
	- \ep  \intO{ \Grad \vr \cdot \Grad \vt \left( \frac{\partial e(\vr, \vt )}{\partial \vt} + \vr \frac{\partial^2 e(\vr, \vt)}{\partial \vr
			\partial \vt }\right)} 
	+ \ep \intO{\Grad \vtB \cdot \Grad \vr \left(s(\vr, \vt) + 
		\vr \frac{\partial s(\vr, \vt)}{\partial \vr}  \right) }	
	\br& + \ep \intO{\vtB \Grad \vr \cdot 
		\Grad \vt \left(\vr \frac{\partial^2 s(\vr, \vt)}{\partial \vr \partial \vt} + \frac{\partial s(\vr, \vt)}{\partial \vt}\right)}.
	\label{p30}
\end{align}

The inequality \eqref{p30} is of the same form as its counterpart 
in \cite[Chapter 6, Section 6.2]{FeiNovOpen} and gives rise to 
the uniform bounds necessary to perform the limits $\ep \to 0$ and 
$\delta \to 0$. 

\subsection{Strict positivity of the approximate temperature}

Going back to the approximate internal energy balance \eqref{p22}
we may apply the temperature minimum principle established in Proposition \ref{TT1} as soon as 
\begin{equation} \label{p31}	
2 \frac{\partial e(\vr, \vt)}{\partial \vr} + \vr \frac{\partial^2 e(\vr, \vt)}{\partial \vr^2}  \geq 0.
\end{equation}
Since 
\[
2 \frac{\partial e(\vr, \vt)}{\partial \vr} + \vr \frac{\partial^2 e(\vr, \vt)}{\partial \vr^2} = \frac{\partial}{\partial \vr}
\left( e(\vr, \vt) + \vr \frac{\partial e(\vr, \vt)}{\partial \vr}\right),\ e(\vr, \vt) + \vr \frac{\partial e(\vr, \vt)}{\partial \vr} = \frac{3}{2} \vt P'(Z),\ Z = \frac{\vr}{\vt^{\frac{3}{2}}},
\]
we need $P = P(Z)$ to be convex. 

As the lower bound on the temperature established in Proposition \ref{TT1}
depends only on the structural constants in the constitutive equations and the initial/boundary data, the same property will be inherited by any weak solution 
resulting from the limit process $\ep \to 0$, $\delta \to 0$ in the approximation scheme \eqref{p20}--\eqref{p22}. We have shown the following result. 

\begin{Theorem}  \label{TP1}{\bf (Temperature minimum principle for weak solutions).} 
	
	Let $\Omega \subset R^3$ be a bounded domain of class $C^{2 + \nu}$. Suppose the thermodynamic functions $p$, $e$, and $s$ are 
	given by \eqref{p7}--\eqref{p9}, where $P \in C^1[0, \infty)$ is a convex function, 
	\begin{equation}\label{p32}
	P(0) = 0,\ P'(0) > 0, \ 0 < \frac{5}{3}P(Z) - P'(Z)Z < \Ov{P} 
	\ \mbox{for all}\ Z > 0,\ \frac{P(Z)}{Z^{\frac{5}{3}}} \to 
	p_\infty > 0\ \mbox{as}\ Z \to \infty.
		\end{equation}
	Let the transport coefficients satisfy 	
	\begin{align}  
		0 < \underline{\mu}\max \{ 1, \vt \} &\leq \mu(\vt) \leq 
		\Ov{\mu}(1 + \vt),\ |\mu' (\vt)| \aleq 1,
		\br 
		0< \underline{\eta} \min\{ 1, \vt \}  &\leq \eta(\vt) \leq 
		\Ov{\eta}(1 + \vt), \br
		0 < \underline{\kappa}\left( \vt^{\frac{1}{2}} + \vt^\beta\right)	
		&\leq \kappa (\vt) \leq \Ov{\kappa} \left( \vt^{\frac{1}{2}} + \vt^\beta\right).	
		\nonumber
	\end{align}
Finally, suppose the initial/boundary data belong to the class 	
\[
\vr_0 > 0,\ \vt_0 > 0,\ \intO{\left(\frac{1}{2} \vr_0 |\vu_0|^2 + \vr_0 e(\vr_0, \vt_0) - \vtB \vr_0 s(\vr_0, \vt_0)\right)} < \infty,
\]
\[
\inf_{\Omega} \vt_0 \geq \underline{\vt} > 0, 
\inf_{\partial \Omega} \vtB \geq \underline{\vt} > 0, 
\]	
and
\[
\vc{g} \in L^\infty((0,T) \times \Omega; R^3).
\]

Then the Navier--Stokes--Fourier system \eqref{i1}--\eqref{i3}, 
with the boundary conditions \eqref{i4} (with $\vuB = 0$), and the 
initial conditions \eqref{i5} admits a weak solution 
$(\vr, \vt, \vu)$ in $(0,T) \times \Omega$ in the sense specified 
in Section \ref{WS}. Moreover, there exists $M > 0$ depending 
only on the structural constants $(\underline{\eta}, \underline{\kappa}, \Ov{\kappa}, \Ov{P})$ such that 
\begin{equation} \label{us}
\vt (t, x) \geq \mathcal{K}^{-1} \left(\left( \frac{2}{\mathcal{K}(\underline{\vt})} + M t \right)^{-1} \right) \ \mbox{for a.a.}\ (t,x) \in (0,T) \times \Omega, 
\ \mathcal{K}(\vt) = \int_0^\tau \kappa (s) \ \D s.
\end{equation}
	\end{Theorem} 

\begin{proof}
	
	The uniform estimate \eqref{us} at the level of the approximate solutions satisfying the scheme \eqref{p20}--\eqref{p22} follows from Proposition \ref{TT1}. The successive limits $\ep \to 0$, $\delta \to 0$ that give rise to a weak solution can be performed in the same way as in \cite[Chapter 12]{FeiNovOpen}.
	
	\end{proof}
	
\begin{Remark} \label{rrr1}
An iconic example of $P$ satisfying \eqref{p32} is 
\[
P(Z) = \left\{  \begin{array}{l} Z , \ \mbox{if}\ Z \leq \Ov{Z}, \\ A Z^{\frac{5}{3}} + B \ \mbox{if} \ Z \geq \Ov{Z},  \end{array}\right. , 
\Ov{Z} > 0,\ A = \frac{3}{5} {\Ov{Z}}^{- \frac{2}{3}},\ B = \frac{2}{5} \Ov{Z}.
\]	
Roughly speaking, the gas obeys the standard Boyle--Mariotte law in the non--degenerate area $\frac{\vr}{\vt^{\frac{3}{2}}} = Z \leq \Ov{Z}$, 
while it behaves like a Fermi gas dominated by free electrons in the degenerate area $Z > \Ov{Z}$.
	\end{Remark}	
	
\section{Concluding remarks}
\label{C}

In the analysis of the approximate system \eqref{p20}--\eqref{p22}, 
it is convenient to work with smooth functions, in particular the 
function $P$ determining the pressure, internal energy, and entropy 
should be smooth. We show that smoothing does not perturb considerably
the structural properties of $P$, in particular, the constants $p_\infty$ and $\Ov{P}$. Indeed consider 
\[
P_\delta (Z) = \int_{Z- \delta}^{Z + \delta} P(y) \zeta(y - Z) \ \D y,\ Z \geq \delta, 
\]
where $\zeta$ is a regularizing kernel, meaning 
\[
\zeta \in C^\infty(R),\ \zeta \geq 0,\ \zeta (s) = \zeta(-s), 
{\rm supp}[\zeta] \subset [-\delta, \delta],\ \int_R \zeta (s) \ \D s = 1.
\]

\begin{itemize}
	\item As regularization commutes with differentiation, we get 
\[
P(Z) \ \mbox{convex} \ \Rightarrow \ P_\delta(Z) \ \mbox{convex} 
\ Z \geq \delta.
\]	

\item As $P$ is non--decreasing, we have  
\[
P(Z- \delta) \leq P_\delta(Z) \leq P(Z + \delta), 
\]
in particular 
\[
\frac{P(Z)}{Z^{\frac{5}{3}}} \to p_\infty 
\ \Rightarrow \ \frac{P_\delta(Z)}{Z^{\frac{5}{3}}} \to p_\infty.
\]

\item 
We have 
\[
\frac{5}{3} P_\delta(Z) - P_{\delta}'(Z) Z = 
\frac{5}{3} P_\delta(Z) - [P'(Z) Z]_{\delta} 
+ [P'(Z) Z]_{\delta}  - P_{\delta}'(Z) Z.
\]
Furthermore, as $P$ is convex, 
\begin{align} 
[P'(Z) Z]_{\delta}  - P_{\delta}'(Z) Z &= 
\int_{Z - \delta}^{Z + \delta} P'(y) (y - Z) \zeta (y- Z) \ \D y \br 
&\geq \int_{Z - \delta}^{Z + \delta} (P(y) - P(Z)) \zeta (y- Z) \ \D y
\br&= \int_{Z - \delta}^{Z + \delta} P(y) \zeta (y- Z) \ \D y - 
P \left( \int_{Z - \delta}^{Z + \delta} y \zeta (y- Z) \ \D y  \right)
\geq 0,
\nonumber
	\end{align}
where the last inequality is Jensen's inequality. 
We have shown 
\[
\frac{5}{3} P(Z) - P'(Z) Z > 0 \ \Rightarrow 
\frac{5}{3} P_\delta(Z) - P_{\delta}'(Z) Z > 0.
\]

\item 
Finally, suppose 
\begin{equation} \label{konec}
\frac{5}{3} P(Z) - P'(Z) Z \leq \Ov{P}. 
\end{equation}
Similarly to the above, we have 
\begin{align}
&
\frac{5}{3}  P_\delta(Z) - P'_\delta(Z) Z = 
\frac{5}{3} P_\delta(Z) - [P'(Z) Z]_{\delta} 
+ [P'(Z) Z]_{\delta}  - P_{\delta}'(Z) Z \br &\leq 
K + \int_{Z - \delta}^{Z + \delta} P'(y) (y - Z) \zeta (y- Z) \ \D y
= K + \int_{Z - \delta}^{Z + \delta} (P'(y) - P'(Z)) (y - Z)) \zeta (y- Z) \ \D y,
\nonumber
\end{align}
where
\begin{align} 
\int_{Z - \delta}^{Z + \delta} (P'(y) - P'(Z)) (y - Z)) \zeta (y- Z) \ \D y \leq \delta \int_{Z-\delta}^{Z + \delta} P''(s)\ \D s
\nonumber
\end{align}

On the one hand, obviously, 
\[
\int_{Z - \delta}^{Z + \delta} (P'(y) - P'(Z)) (y - Z)) \zeta (y- Z) \ \D y \leq \delta \sup_{Z \leq \Ov{Z}} P'(Z) \ \mbox{for}\ 
Z \leq \Ov{Z}.
\]
On the other hand,  we may suppose
\[
{\rm ess} \limsup_{Z \to \infty } {P''(Z)} < \infty, 
\]
which does not seem very restrictive in view of \eqref{konec}. 
Then we can conclude 
\[
\frac{5}{3} P(Z) - P'(Z) Z \leq \Ov{P} \ \Rightarrow\ 
\frac{5}{3} P_\delta(Z) - P'_\delta(Z) Z \leq \Ov{P} + O(\delta),
\ O(\delta) \to 0 \ \mbox{as}\ \delta \to 0.
\]

	\end{itemize}


\def\cprime{$'$} \def\ocirc#1{\ifmmode\setbox0=\hbox{$#1$}\dimen0=\ht0
	\advance\dimen0 by1pt\rlap{\hbox to\wd0{\hss\raise\dimen0
			\hbox{\hskip.2em$\scriptscriptstyle\circ$}\hss}}#1\else {\accent"17 #1}\fi}

\end{document}